\documentclass[12pt, a4paper]{amsart}
\usepackage{amsmath,amssymb,amscd,amsfonts}
\newtheorem{theorem}{Theorem}[section]
\newtheorem{rem}[theorem]{Remark}

\newtheorem{lemma}[theorem]{Lemma}

\newtheorem{corollary}[theorem]{Corollary}

\def\det{\mathop{\rm det}\nolimits}
\def\Ricci{\mathop{\rm Ricci}\nolimits}

\def\dbar{\overline\partial}
\def\ddbar{\partial\overline\partial}

\def\cG{{\mathcal G}}

\def\cE{{\mathcal E}}

\def\cC{{\mathcal C}}
\def\cM{{\mathcal M}}
\def\cK{{\mathcal K}}
\let\ol=\overline

\let\ep=\varepsilon

\def\bC{{\mathbb C}}
\def\bR{{\mathbb R}}

\begin{document}
\title[Remark]{
Approximation of weak geodesics and \\
subharmonicity of Mabuchi energy}

\author{XiuXiong Chen, Long Li, Mihai P\u aun}

\address{Department of Mathematics\\
 Stony Brook University \\
 NY, USA.} 
 \email{xiu@math.sunysb.edu}

\address{Department of Mathematics and Statistics\\
McMaster University\\
1280 Main Street West\\
Hamilton, ON  L8S 4K1\\
Canada.}
\email{lilong@math.mcmaster.ca}

\address{Korea Institute for Advanced Study \\
School of Mathematics\\
85 Hoegiro, Dongdaemun-gu, Seoul 130-722, Korea.}
\email{paun@kias.re.kr}

\maketitle

\section{Introduction}

In a recent paper \cite{BB}, R. Berman and B. Berndtsson established the convexity of the Mabuchi energy functional
$\cM$ along the so-called \emph{weak geodesics}, answering (affirmatively) to a conjecture proposed 
by the first named author of this article. 
Given two K\"ahler metrics in the same cohomology class, it is well-known (cf. \cite{lempert1}, \cite{lempert2}) that in general one cannot find a
\emph{smooth geodesic} connecting them: this is a major source of difficulties while dealing e.g. with the aforementioned convexity question.

In this article we explore in a systematic way two techniques of approximation of weak geodesics.
The first and most natural one is given by the \emph{$\ep$-geodesics}, obtained in \cite{XX}. The second one
consists in using a fiberwise approximation of weak geodesics via a family of well-chosen Monge-Amp\`ere 
equations. A corollary of the second technique is an alternative 
proof of the result in \cite{BB}. 
Roughly speaking, our proof can be seen as a ``global version" of the 
local Bergman kernel arguments, so morally we follow the original ideas of \cite{BB}; nevertheless, we feel that
our approach might be useful in other contexts. For example, the method we are using here allows us to establish 
the convexity of $\cM$ more \emph{directly} than in the original article, where the 
first step is to show convexity of $\cM$ in the sense of distributions. 

We equally infer 
that the 
Mabuchi functional is \emph{continuous} up to the boundary when 
evaluated on a weak geodesic. The proof of this second statement is based on 
semi-continuity properties of the entropy functional.

Another theorem we will establish here is the almost-convexity of the regularized 
Mabuchi energy along the $\ep$-geodesics cf. \cite{XX}. Actually, our hope is that this latter 
result could be also used in order to provide a proof of the convexity of $\cM$ along weak geodesics.
We refer to the comments at the end of this note for further support concerning this belief.

This article is organized as follows. 
We start by recalling the important result of Xiuxiong Chen in \cite{XX} concerning the existence of 
$\cC^{1,1}$ solutions of the MA equation describing the geodesic between two K\"ahler 
metrics. After a preliminary discussion about the strategy of the proof,
the convexity and the continuity of $\cM$ are obtained in section 4 via the approximation procedure mentioned above.
Finally, the convexity of $\cM$ along $\ep$-geodesics and some other results/expectations 
are treated in section 5.

\section{Geodesics}\label{geo}

Let $X$ be a compact K\"ahler manifold; we denote by $\cK$ its K\"ahler cone.
Let $\{\omega\}\in \cK$ be a K\"ahler class of $X$; the notation above means
that the representative $\omega$ is non-singular and definite positive. Let $\omega_0, \omega_1\in \{\omega\}$ be two positive definite 
representatives of the same cohomology class. A \emph{weak geodesic} between $\omega_0$ and $\omega_1$ is a semi-positive definite
(1,1)-current 
\begin{equation}\label{equa1}  
\cG:= \omega+ dd^c\varphi
\end{equation}
on the product $X\times \Sigma$ of the manifold $X$ with the annulus $\Sigma\subset \bC$, such that the following requirements are 
satisfied.

\begin{enumerate}

\item[(a)] The function $\varphi$ is $\cC^{1,1}$ on $X\times \Sigma$; in particular, the coefficients of $\cG$ are bounded.
\smallskip

\item[(b)] We have 
$\displaystyle \cG^{n+1}= 0.$
\smallskip

\item[(c)] The current $\cG$ is rotationally invariant, and it equals
$\omega_0$ and $\omega_1$ on the boundary of $\Sigma$, respectively. 
\end{enumerate}

The existence of $\cG$ with the properties stated above was first established in \cite{XX}, together with important complements in \cite{Blocki}.

\section{What is to be proved}\label{what}

Let $u$ be a smooth function on $X$, such that $\omega_u:= \omega+ dd^cu$ is a K\"ahler metric. 
The \emph{energy} functional is given by the expression

\begin{equation}\label{equa2002}
\cE(u):= \sum_{j=0}^n\int_Xu\omega_u^{n-j}\wedge \omega^j.
\end{equation}
Given a (1,1)-form $\alpha$, one introduces the following version of the energy functional
\begin{equation}\label{equa2003}
\cE^\alpha(u):= \sum_{j=0}^{n-1}\int_Xu\omega_u^{n-j-1}\wedge \omega^j\wedge\alpha.
\end{equation}
For a smooth path $\omega_t:= \omega+ dd^cu_t$ of K\"ahler metrics depending on the parameter $t\in \Sigma$, one rapidly computes 
\begin{equation}\label{equa2004}
dd^c\cE(t)= \int_{X}\Omega^{n+1}, \quad dd^c\cE^\alpha(t)= 
\int_{X}\Omega^{n}\wedge\alpha,
\end{equation}
where $\Omega:= \omega+ dd^cu$ is a (1,1)-form on $X\times \Sigma$, and the integration is understood as the push-forward of an $(n+1,n+1)$ form to $\Sigma$ (we use the same notation for $\omega$ and its inverse image on $X\times\Sigma$). An important observation (cf. \cite{Chen2}, \cite{BB} and the references therein) is that the equalities (\ref{equa2004}) still holds true \emph{in the sense of distributions} 
if the path $(u_t)$ is only assumed to be continuous.

The \emph{Mabuchi functional} $\cM$ along $\cG$ is defined as follows

\begin{equation}\label{equa2001}
\cM(t) = \frac{S}{n+1} \cE(\varphi_t) - \cE^{Ric_{\omega}} (\varphi_t)  + 
\int_X \log\frac{\cG_{t}^n}{\omega^n} \cG_t^n
\end{equation}
where $\cG_t:= \omega+ dd^c\varphi_t$
is the restriction of $\cG$ to the slice $X\times \{t\}\subset X\times \Sigma$, and $S$ is the average of the scalar curvature of $(X, \omega)$. Unlike the original definition of 
$\cM$, the expression
(\ref{equa2001}) first introduced in \cite{Chen2}, has a meaning even if the regularity of $\varphi$ is only $\cC^{1,1}$. 
We recall further that the convexity of 
$\cM$ along weak geodesics was conjectured in \cite{Chen2}. 

Ideally, the convexity of $\cM$ would follow provided that one is able to
produce the following objects.
\smallskip

\noindent Let $(\Theta_\ep)_{\ep> 0}\subset \{\omega\}$ be a family of closed positive (1,1) currents on $X\times \Sigma$, such that for each positive 
$\ep$ we have.

\begin{enumerate}

\item[(a)] The potential $\phi_\ep$ of each $\Theta_\ep$ is of class $\cC^{1,1}$, and it converges to $\varphi$ locally uniformly on $X\times \Sigma$.
\smallskip

\item[(b)] The logarithm of the fiberwise determinant of $\Theta_\ep$ is of class $\cC^{1}$.
\smallskip

\item[(c)] The determinant of 
$\Theta_\ep|_{X\times \{t\}}$ converges a.e. to 
$\cG^n|_{X\times \{t\}}.$ 
\end{enumerate} 
\medskip

\noindent As explained in \cite{BB}, it is enough show that we can find $(\Theta_\ep)_{\ep> 0}$ as above, such that moreover
we have
\begin{equation}\label{equa100}
dd^c\log\frac{\Theta_\ep^n}{\omega^n}\wedge \cG^n\geq \Ricci_{\omega}
\wedge \cG^n
\end{equation}
where the quantity $\displaystyle \log\frac{\Theta_\ep^n}{\omega^n}$ denotes (slightly abusively) a function on $X\times \Sigma$. 
Indeed, given the relations (\ref{equa2004}), the inequality (\ref{equa100}) implies the convexity 
in the sense of distributions
of the following functional
\begin{equation}\label{equa2005}
\cM(\ep, t):= \frac{S}{n+1} \cE(\varphi_t) - \cE^{Ric_{\omega}} (\varphi_t)  + 
\int_X \log\frac{\Theta_{\ep}^n}{\omega^n} \cG_t^n.
\end{equation}
By the condition (b) the functional $\cM(\ep, t)$ is continuous; therefore its convexity in weak sense implies 
convexity in usual sense. The condition (c) would imply the same property for $\cM$, by letting $\ep\to 0$. 
\medskip

\noindent An excellent candidate for the family $\Theta_\ep$ would be the approximation of $\cG$ contained in the 
proof of X.X. Chen, i.e. the \emph{$\ep$-geodesics}. Indeed, the properties (a) and (b) are direct consequences of \cite{XX}, and the inequality (\ref{equa100}) can be checked to hold true on the set 
$\Lambda_{A, \ep}$ where 
$\Theta_\ep$ is uniformly bounded from below by $\exp (-A)\omega$ via a direct computation (this would be enough to conclude, by letting $A\to \infty$). However, it does not seem to be so easy to establish 
the crucial property (c) (we refer the the paper \cite{CT}, section 6, in order to have a glimpse at  the 
difficulties/consequences of such a statement).

In order to overcome this issue, we will consider in the next paragraph a different approximation of $\cG$, obtained by solving a family of MA equations. 

A version of the convexity of $\cM$ along the $\ep$-geodesics will be treated in the last part of our note.

\section{Fiber-wise approximation of $\cG$}

\noindent 
In order to construct the family of currents $(\Theta_\ep)_{\ep> 0}$ with the properties stated in the 
previous section we recall the following result.

\begin{theorem} \label{psh}
{\rm (\cite{MP})} Let $p: X\to Y$ be a holomorphic submersion. We consider  
a semi-positive class $\{\beta\}\in H^{1,1}(X, \bR)$, such that the adjoint class $c_1(K_{X_y})+ \{\beta\}|_{X_y}$
is K\"ahler for
any $y\in Y$.
Then the relative 
adjoint class
$$c_1(K_{X/Y})+ \{\beta\}$$
contains a closed positive current $\Xi$, whose restriction to each fiber $X_y$ is a positive definite form.

\end{theorem}

\medskip

\noindent We specialize here to the case of the trivial submersion
$X\times \Sigma
\to \Sigma$, so the relative canonical bundle equals the inverse image of $K_X$.
\medskip



\noindent As a consequence of the previous result we infer the next statement; we recall that $\cG$ denotes the (weak) geodesic between the two metrics $\omega_0$ and $\omega_1$.

\begin{theorem}\label{coro1}
For each $t\in\Sigma $ and for each $0< \ep\ll 1$, we consider the equation 
\begin{equation}\label{equa1208}
\big(\Theta_{\omega}(K_X)+ {\varepsilon}^{-1}\cG+ dd^c\phi_{t, \ep}\big)^n= \ep^{-n}\exp(\phi_{t, \ep})\omega^n;
\end{equation} 
on $X\times \{t\}$. It has a unique $\cC^{1,1}\big(X\times \{t\}\big)$-solution $\phi_{t, \ep}$; the 
resulting function $\phi_\ep $ on
$X\times \Sigma$ is continuous on the interior points of $X\times \Sigma$, and moreover we have 
\begin{equation}\label{equa1308}
\Theta_{\omega}(K_X)+ {\varepsilon}^{-1}\cG+ dd^c\phi_{\ep}\geq 0
\end{equation}
on the product manifold $X\times \Sigma$.
\end{theorem}

\begin{proof}

\smallskip
We proceed in a very standard manner, namely by an approximation argument.
Let $(\cG_\delta)$ be a family of (1,1)-forms on the open set
$$X\times \Sigma^\prime\subset X\times \Sigma$$
obtained by considering the convolution of the 
potential $\varphi$ with a convolution kernel $K_\delta$, cf. e.g. \cite{DEM1}. Here $\Sigma^\prime$ stands for 
a compact subset of the annulus $\Sigma$.
The resulting (1,1) forms are
approximating our weak geodesic $\cG$, as follows.

\begin{enumerate}

\item[(i)] The forms $\cG_\delta$ are non-singular, and we 
have 
$$\cG_\delta\geq -C\delta(\omega+ \sqrt{-1}dt\wedge d\ol t)$$ on $X\times \Sigma^\prime.$
\smallskip

\item[(ii)] The coefficients of $\cG_\delta$ are uniformly bounded, 
and they converge in $L^p$ norm to the coefficients of $\cG$,
so that if we write 
$$\cG_\delta:= \omega+ dd^c\varphi_\delta$$
then we have 
$$\vert dd^c\varphi_\delta- dd^c\varphi\vert_{L^p(X\times \{t\})}\to 0$$
as $\delta\to 0$, uniformly with respect to $t\in \Sigma^\prime$ and for any $p$.
\smallskip

\item[(iii)] We have $$\sup_{t\in \Sigma^\prime}\Vert \varphi_\delta-\varphi\Vert_{\cC^1(X\times \{t\})}\to 0$$
as $\delta \to 0$. 
\smallskip


\end{enumerate}
\medskip

\noindent The properties (ii) and (iii) of the approximation family $(\cG_\delta)$ hold true thanks to the 
explicit construction of $(\cG_\delta)$ by using the convolution of the potential of $\cG$ with a regularizing kernel (\cite{DEM1}), combined with the  fact that the potential of $\cG$ is $\cC^{1,1}$.

For each $\eta:= (\ep, \delta)$ such that $\ep$ and $\delta$ are positive and small enough 
we define the semi-positive form
\begin{equation}\label{equa9}
\beta_{\eta}:= \frac{1}{\ep} \big(\cG_\delta+ C\delta(\omega+ \sqrt{-1}dt\wedge d\ol t)\big). 
\end{equation}
The class $c_1(K_X)+ \{\beta_{\eta}\}|_{X\times \{t\}}$ is clearly K\"ahler; by the classical result of S.-T. Yau
we infer that there exists
a function $\phi_{t, \eta}$ such that we have
\begin{equation}\label{equa10}
\Theta_{\omega}(K_X)+ \beta_{\eta}|_{X\times \{t\}}+ dd^c\phi_{t, \eta}>0
\end{equation}
together with
\begin{equation}\label{equa11}
\big(\Theta_{\omega}(K_X)+ \beta_{\eta}+ dd^c\phi_{t, \eta}\big)^n= \ep^{-n}\exp(\phi_{t, \eta})\omega^n
\end{equation}
on the fiber $X\times \{t\}$.
\medskip

According to the proof of Theorem \ref{psh}, we infer that
\begin{equation}\label{equa1408}
\Xi_{\eta}: = \Theta_{\omega}(K_X)+ \beta_{\eta}+ dd^c\phi_{\eta}>0,
\end{equation}
as a smooth positive $(1,1)$ form on $X\times \Sigma^\prime$--actually, this is the only reason why we need to consider the regularization of 
$\cG$ with respect to the parameter ``$t$'' as well.
\smallskip

\noindent We show next that the family $\phi_{\eta}$ is equicontinuous \emph{for each $\ep$ fixed}; prior to this, we introduce the a few
notations. 

Let $\tau_\eta:= \ep \phi_\eta+ \varphi_\delta$; it is a smooth function defined on $X\times \Sigma^\prime$.
Then for each $t_0, t_1\in \Sigma^\prime$ we have
\begin{equation}\label{equa1508}
\big(\Psi_{t_0, \eta}+ dd^c(\tau_{\eta}(t_1)- \tau_{\eta}(t_0))\big)^n= e^{\frac{1}{\ep}(\tau_{\eta}(t_1)- \tau_{\eta}(t_0)- 
\varphi_{\delta}(t_1)+ \varphi_{\delta}(t_0))}\Psi_{t_0, \eta}^n
\end{equation}
where $\displaystyle \Psi_{t, \eta}:= \ep\Xi_{\eta}|_{X\times \{t\}}$.

The maximum principle, combined with the property (iii) above (concerning the uniformity properties 
of the sequence $\varphi_\delta$ on $X\times \Sigma^\prime$) shows that we have
\begin{equation}\label{equa1608}
\sup_{x\in X}|\phi_\eta(t_0, x)- \phi_\eta(t_1, x)|\leq C\ep^{-1}|t_0- t_1|
\end{equation}
The same kind of arguments (i.e. the maximum principle applied on fibers $X\times\{t\}$)
show that in fact we have 
 \begin{equation}\label{equa1708}
|\phi_\eta(t_0, x_0)- \phi_\eta(t_1, x_1)|\leq C\ep^{-1}(|t_0- t_1|+ {\rm dist(x_0, x_1)})
\end{equation}
where $C$ is a constant independent of $\eta=(\ep, \delta)$. This proves the claimed equicontinuity.

As a consequence, the limit $\phi_\ep:= \lim_{\delta\to 0}\phi_{\ep, \delta}$ is continuous on the interior of 
$X\times \Sigma$, and we have
 \begin{equation}\label{equa1808}
\Theta_{\omega}(K_X)+ \frac{1}{\ep}\cG+ dd^c\phi_{\ep}\geq 0.
\end{equation}
which finishes the proof of Theorem \ref{coro1}, except for the $\cC^{1,1}\big(X\times \{t\}\big)$-regularity of the 
solution $\phi_{t, \ep}$; this will be treated in out next result.
 \end{proof}
\medskip

\noindent In our next statement we will use another type of regularization of the geodesic $\cG$, borrowed from \cite{BD1}.
We define $\cG^\prime_{\delta}:= \omega+ dd^c\varphi_\delta$ on $X\times \Sigma$, such that for each $t\in \Sigma$,
the function $\varphi_\delta$ is the regularization of $\varphi|_{X\times \{t\}}$ by a global convolution kernel.
The properties of the resulting form which will be relevant for us are as follows.
\begin{enumerate}

\item[(1)] The forms $\cG^\prime_\delta$ are non-singular when restricted to each fiber $X\times \{t\}$, and 
moreover 
there exists a constant $C> 0$ such that we have
$$\vert dd^c\varphi_\delta\vert\leq C$$
for any $\delta> 0$.
We equally 
have 
$$\cG^\prime_\delta\geq -C\delta(\omega+ \sqrt{-1}dt\wedge d\ol t)$$ on $X\times \Sigma.$
\smallskip

\item[(2)] The coefficients of $\cG^\prime_\delta$ are uniformly bounded, 
and they converge in $L^p$ norm to the coefficients of $\cG$,
so that if we write 
$$\cG^\prime_\delta:= \omega+ dd^c\varphi_\delta$$
then we have 
$$\vert dd^c\varphi_\delta- dd^c\varphi\vert_{L^p(X\times \{t\})}\to 0$$
as $\delta\to 0$, uniformly with respect to $t\in \Sigma$ and for any $p$.
\smallskip

\item[(3)] We have $$\lim\Vert \varphi_\delta-\varphi\Vert_{\cC^0(X\times \Sigma)}= 0$$
as $\delta \to 0$. 
\smallskip


\end{enumerate}
\medskip
We consider the Monge-Amp\`ere equation
\begin{equation}\label{equa1908}
\big(\Theta_{\omega}(K_X)+ \beta^\prime_{\eta}+ dd^c\phi_{t, \eta}\big)^n= \ep^{-n}\exp(\phi_{t, \eta})\omega^n
\end{equation}
on the fiber $X\times \{t\}$, where
\begin{equation}\label{equa2008}
\beta_{\eta}:= \frac{1}{\ep} \big(\cG^\prime_\delta+ C\delta(\omega+ \sqrt{-1}dt\wedge d\ol t)\big). 
\end{equation}

\noindent The regularity/uniformity properties of the functions $(\phi_{t, \eta})_{\eta}$ are stated in the following result.

\begin{theorem} \label{MA}
The following assertions hold true.

\begin{enumerate}

\item[(a)] For each fixed $\ep$, the family $\phi_{\eta}$ obtained by piecing together the 
fiber-wise solutions $\phi_{t, \eta}$ is equicontinuous.
\smallskip

\item[(b)] There exists a constant $C> 0$, independent of $\eta$ 
such that
$$ \sup_X\phi_{\ep, \delta} \leq C, \ \ -\ep \inf_X\phi_{\ep,\delta} \leq C,   \ \  |\ep dd^c\phi_{\ep, \delta}|\leq C.$$
on the fiber $X\times \{t\}$.
\smallskip

\item[(c)] Therefore for each fixed $\ep> 0$, we can extract a limit 
$$\lim_{\delta\to 0}\phi_{\ep, \delta}= \phi_\varepsilon,$$
strongly in $\cC^0$, where the restriction of $\phi_\varepsilon$ to $X\times \{t\}$ is the unique 
$\cC^{1,1}$ solution of the degenerate Monge-Amp\`ere equation
\begin{equation}\label{equa12}
\big(\Theta_{\omega}(K_X)+ {\varepsilon}^{-1}\cG+ dd^c\phi_{t, \ep}\big)^n= \ep^{-n}\exp(\phi_{t, \ep})\omega^n
\end{equation}
on the fiber $X\times \{t\}$. 

\smallskip

\item[(d)] The measures 
$$\exp(\phi_{t, \ep})\omega^n$$
are converging to $\displaystyle \cG|_{X\times \{t\}}^n$ weakly in $L^p$ for any $p$, as $\ep\to 0$.

\end{enumerate}

\end{theorem}

\begin{proof} Except maybe for the point (d), the arguments of the proof 
rely on basic results in MA theory, so we will be very sketchy. 

The point (a) was basically discussed during the proof of Theorem \ref{coro1}. In addition, we remark that by th esame procedure we obtain the equicontinuity continuity of $(\phi_\eta)$ on $X\times\Sigma$ up to the boundary of $X\times\Sigma$: this is a consequence of the 
 property (2) of $\cG_\delta^\prime$. 

Concerning the point (c), the upper bound of the potentials is a consequence of the maximum 
principle since we can rewrite the equation (\ref{equa1908}) as 
\begin{equation}\label{equa112.5}
\big(\varepsilon\Theta_{\omega}(K_X) + \cG_{\delta} + dd^c (\varepsilon\phi_{t,\eta})\big)^n = \exp(\phi_{t,\eta})\omega^n,
\end{equation}
then take $\phi_{t,\eta}(p) = \max_X \phi_{t,\eta}$, which implies $dd^c \phi_{t,\eta}(p) \leqslant 0$. But notice the form 
$ \varepsilon\Theta_{\omega}(K_X)+ \cG_{\delta} $ is strictly positive at the point $p$ thanks to the inequality (\ref{equa10}), hence 
\[
\phi_{t,\eta} \leqslant \phi_{t,\eta}(p) \leqslant \log\frac{(\varepsilon \Theta_{\omega}(K_X) + \cG_{\delta})^n}{\omega^n}(p)\leqslant C.
\]

On the other hand, the lower bound is obtained by considering 
\begin{equation}\label{equa112.6}
(\omega+ \varepsilon\Theta_{\omega}(K_X) + dd^c \tau_{t,\eta} )^n = e^{\frac{1}{\epsilon}(\tau_{t,\eta} - \varphi_{\delta})}\omega^n 
\end{equation}
where $\tau_{t,\eta} = \varepsilon\phi_{t,\eta} + \varphi_{\delta}$, so by choosing $\tau_{t,\eta}(q) = \min_X \tau_{t,\eta}$, the minimum principle says 
\[
\ep\phi_{t,\eta} + \varphi_{\delta}\geq
\varepsilon\log\frac{(\omega+\varepsilon \Theta_{\omega}(K_X))^n}{\omega^n}(q) + \varphi_{\delta}(q)\geqslant - \varepsilon C + \varphi_{\delta}(q), 
\]
hence $\ep\phi_{t,\eta} \geqslant -C$ for some uniform constant $C$. Now the bound of $dd^c \tau_{t,\eta}$ follows from the usual Laplacian estimate for the Monge-Amp\'ere equations, cf. \cite{Blocki}. For fixed $\eta$, let $\omega_{\varepsilon} = \omega+ \varepsilon\Theta_{\omega}(K_X)$, and notice that it is a smooth non-degenerate approximation of $\omega$, then we can write $\Delta\tau = g^{j\bar{k}}_{\ep} \tau_{j\bar{k}}$, and set 
\[
\alpha: = \log (n+\Delta \tau) - A \tau,
\]
where $A>0$ is some constant determined later. Consider the point $p$ where the maximum of $\alpha$ is obtained, and take 
$u = G_{\ep} + \tau$, where the function $G_{\ep}$ is the local potential of $\omega_{\ep}$ in a normal coordinate ball of $p$, then the standard
maximum principle argument implies at the point $p$
\begin{equation}\label{equa112.7}
0 \geqslant \frac{1}{\Delta u}\big\{ -B\Delta u\sum_{p}\frac{1}{u_{p\bar{p}}} +
 \Delta\phi_{t,\eta} - S \big\}
+ A\sum_{p} \frac{1}{u_{p\bar{p}}} - nA,
\end{equation}
where $B>0$ is the lower bound of bisectional curvature, and $S$ is the upper bound of scalar curvature of $(X, \omega_{\ep})$. Now take $A =B$, then we have 
\[
nB \Delta u + S \geq \frac{1}{\ep} \Delta (\tau - \varphi_{\delta}), 
\]
hence
\[
n+\Delta \varphi_{\delta} +S\ep \geq (1 - \ep n B) \Delta u,
\]
and finally 
\[
\Delta u (p) \leq \frac{C}{ 1 - \ep n B} < 2C
\]
for some uniform constant $C$, when $\ep < \frac{1}{2nB}$. Then suppose there is some uniform constant $C'$ such that $osc(\tau) < C'$, we infer
\[
\Delta u \leq 2C e^{2BC'}  
\]

The statement (d) is a consequence of the elliptic regularity results, cf. \cite{GT}. Moreover, since the solution $\phi_{t, \ep}$ belongs to the space $\cC^{1,1}$ we infer that the equality 
(\ref{equa12}) holds almost everywhere  on $X\times \{t\}$ (in the sense of $L^{\infty}$ functions).

In order to establish the point (e), we re-write the equation (\ref{equa12}) as follows
\begin{equation}\label{equa113}
(\omega+ \ep \Theta_{\omega}(K_X)+ dd^c\tau_\ep)^n= e^{\frac{1}{\ep}(\tau_\ep- \varphi)}\omega^n
\end{equation}
where we recall that $\cG= \omega+ dd^c\varphi$. In the relation (\ref{equa113}), we denote 
$\tau_\ep:= \ep\phi_\ep+ \varphi$, so we will be done if we can prove that 
$$\ep\phi_\ep\to 0.$$
In any case, thanks to the estimates (c), there exists some function 
$\rho\in \cC^{1,1}$ such that we have

\begin{equation}\label{equa42}
\ep\phi_\ep\to \rho
\end{equation}
strongly in $\cC^{1, \alpha}$ for any $\alpha< 1$ as $\ep\to 0$, so that the limit $\tau$ extracted from 
$\tau_\ep$ verifies the inequality 
\begin{equation}\label{equa43}
\tau\leq \varphi.
\end{equation}
as it is clear from the second part of estimate (c). 

In particular, the convergence statement in (\ref{equa42}) implies that we have
\begin{equation}\label{equa143}
(\omega+ \ep \Theta_{\omega}(K_X)+ dd^c\tau_\ep)^n\to (\omega+ dd^c\tau)^n
\end{equation}
is weak sense in $L^p$ for any $p$ (this is due to the fact all currents in the concern have uniformly bounded $L^{\infty}$ coefficients ).
 
Let $\Omega_\delta:= \{ \tau< \varphi-\delta \}$; it is an open subset of $X$, and we claim that $\displaystyle \cG^n|_{\Omega_\delta}=0$, for each $\delta> 0$.
Indeed, by the comparison principle \cite{SK}, \cite{GZ07} we have
\begin{equation}\label{equa44}
\int_{\Omega_\delta}\cG^n\leq \int_{\Omega_\delta}(\omega+ dd^c\tau)^n.
\end{equation} 
However, as we can see from equation (\ref{equa113}), we have 
$\displaystyle \int_{\Omega_\delta}(\omega+ dd^c\tau)^n=0$, 
simply because on the set $\Omega_\delta$ the inequality
$\displaystyle \frac{\tau_\ep-\varphi}{\ep}< -\frac{\delta}{2\ep}$ as soon as $\ep$ is small enough
--remark that this is a consequence of the uniform convergence in (\ref{equa42})--
 so our claim is proved. 
\smallskip

But then it follows that $\displaystyle \cG^n|_{\overline \Omega}=0$, where $\overline\Omega$ is the 
closure of the open set $\{ \tau< \varphi \}\subset X$. Indeed, the set  
$\overline \Omega\setminus \Omega$ has measure zero, and the coefficients of $\cG$ are bounded.
We infer that we have
\begin{equation}\label{equa45}
\cG^n= (\omega+ dd^c\tau)^n
\end{equation}
since the complement of $\ol\Omega$ is an open set where $\tau$ coincides with $\varphi$. 

We invoke next the uniqueness result in \cite{SK} concerning the solutions of MA equations whose right hand side member has a density in $L^p$, for $p> 1$, so that $\rho=0$ and the point (e)
of our lemma follows.
\end{proof}

\begin{rem} {\rm The fiber-wise convergence (e) was proved in \cite{Berman} in a slightly different setting; the main argument in that article relies on the variational approach developed in \cite{BBGZ}. However, we have chosen to give a direct proof here, for the sake of variation.}  
\end{rem}

\subsection{Convexity}
Let $t_0, t_1, t_2\in [0,1]$ be three arbitrary points. By the property (e) combined with a 
result due to Banach-Saks one can find a 
sequence $\ep_k\to 0$ as $k\to \infty$ such that 
$$\frac{1}{k}\sum_{j=1}^k\exp(\phi_{t_p, \ep_j})\to \exp(\varphi|_{X\times \{t_p\}})$$
in $L^1(X\times \{t_p\})$, for each $p= 0, 1, 2$. The sequence $(\ep_k)$ depends on the triple 
$(t_p)$, but fortunately this does not matter for the rest of the proof.

For each $t\in \Sigma$, let $\omega_{t, k}\in \{\omega \}$ be the K\"ahler metric such that 
\begin{equation}\label{equa112}
\omega_{t, k}^n= \frac{1}{k}\sum_{j=1}^k\exp(\phi_{t, \ep_j})\omega^n
\end{equation}

\noindent Let $\cM_k$ be the Mabuchi functional evaluated on the weak geodesic $\cG$, with the entropy term modified by using $\log\omega_{t, k}^n$ instead of $\log\cG^n$, i.e.
\begin{equation}\label{equa12.1}
\cM_k(t) = \frac{S}{n+1} \cE(\varphi_t) - \cE^{Ric_{\omega}} (\varphi_t) + \int_X \log\frac{\omega_{t,k}^n}{\omega^n} \cG^n
\end{equation}
Then we have the following statement.

\begin{lemma}
For each $k\geq 1$ the functional $\cM_k$ is a continuous convex function on $[0,1]$.
\end{lemma}

\begin{proof}
Our first observation is that the functional $\cM_k$ is continuous. Indeed, this is the case
given the regularity we have already established in Theoren \ref{MA} for each $\phi_k$, combined with the stability theorem due to 
S. Kolodziej, cf. \cite{SK}; we do not give further details here.

It is therefore enough to show that $\cM_k$ is convex in weak sense; this boils down to the 
inequality
\begin{equation}\label{equa13}
dd^c\log\Big(\frac{1}{k}\sum_{j=1}^k\exp(\phi_{\ep_j})\Big)\wedge \cG^n\geq 
\Ricci_{\omega}
\wedge \cG^n.
\end{equation}

This is immediately seen to be true, as follows. We have 
\begin{equation}\label{equa14}
dd^c\log\frac{1}{k}\sum_{j=1}^k\exp(\phi_{\ep_j})\geq \sum_{j=1}^k
\frac{\exp(\phi_{\ep_j})}{\sum_{i=1}^k\exp(\phi_{\ep_i})}dd^c\phi_{\ep_j}
\end{equation}
by a direct computation, and moreover
the point (b) of Theorem \ref{MA} shows that we have
\begin{equation}\label{equa15}
-\Ricci_{\omega}+ \ep_j^{-1}\cG+ dd^c\phi_{\ep_j}\geq 0.
\end{equation}
By using this inequality in (\ref{equa14}), we obtain
\begin{equation}\label{equa15}
dd^c\log\frac{1}{k}\sum_{j=1}^k\exp(\phi_{\ep_j})\wedge \cG^n \geq 
\sum_{j=1}^k\frac{\exp(\phi_{\ep_j})}{\sum_{i=1}^k\exp(\phi_{\ep_i})}\Ricci_{\omega}\wedge \cG^n
\end{equation}
which proves the lemma.
\end{proof}
\medskip

\begin{corollary}
The Mabuchi functional $\cM$ is a convex function on the interval $[0,1]$. 
\end{corollary}
\begin{proof}
 Let $t_{p}\in [0, 1]$ be three arbitrary points; we can apply the convexity inequality for each $\cM_k$ corresponding to the points $(t_p)_{p=0, 1, 2}$ and we let $k\to \infty$; the convexity of $\cM$ follows.
\end{proof}
\medskip

\subsection{Continuity at the boundary}
Given that $\cM$ is a convex function, it is automatically continuous on the open interval $]0, 1[$. 
In this subsection we will show that the continuity property holds up to the boundary.

To this end we recall that the entropy functional $H$ is defined as 
\[
H(\varphi) = \int_X f_{\varphi}\log f_{\varphi} d\mu,
\]
where 
\[
f_{\varphi} = \frac{\omega^n_{\varphi}}{\omega^n}
\]
is an $L^{\infty}$ function provided that the potential $\varphi\in\mathcal{C}^{1,{1}}$, and the probability measure $d\mu$
equals $\omega^n$. In order to study the semi-continuity property of $H$, we will follow \cite{CT},
\begin{lemma} 
\label{Chen}
Let $\varphi_i, \varphi$ and $f_i, f$ be as above, and suppose $f, f_i$ are uniformly bounded non-negative functions, such that $f_i\rightarrow f$ weakly in $L^1$, then 
\[
\lim_{i} \int_X ( f_i\log f_i - f\log f) d\mu \geqslant 0
\]
\end{lemma}
\begin{proof}
First assume $f$ has positive lower bound, i.e. $f> \delta$ for some small $\delta>0$, since we can replace $f$ by $f+\delta$, and let $\delta$ converges to zero. Now put 
$F_i(t) = \mathcal{F} (t f_i + (1-t)f) = \mathcal{F}(at+b)$, where $a = f_i -f$ and $b = f$, then 
\[
F'_i(t) = a (\log u_t +1),
\]
where $u_t = t f_i +(1-t) f$, and 
\[
F''_i(t) = \frac{a^2}{u_t} \geqslant \frac{a^2}{C},
\]
for some constant uniform constant $C$ such that $f_i$ and $f$ are smaller than $C$. Hence 
\begin{equation}\label{Chen1}
\int_X ( f_i\log f_i - f\log f) d\mu = \int_X ( \int_0^1\int_0^t F''(s)dsdt + \int_0^1 F'(0)dt ) d\mu 
\end{equation}
\[
\geq\frac{1}{C}\int_X (f_i - f)^2 d\mu + \int_X F'_i(0) d\mu.
\]
However,
\begin{equation}\label{Chen2}
\lim_{t\rightarrow 0} \int_X F'_i(t) d\mu = \lim_{t\rightarrow 0} \int_X (f_i - f) \{ \log (t f_i + (1-t)f) + 1 \} d\mu
\end{equation}
\[
= \int_X (f_i -f )(\log f +1 ) d\mu,
\]
where we have used the fact $f_i \rightarrow f$ weakly. Then obviously
\[
\lim_{i}\lim_{t\rightarrow 0} \int_X F'_i d\mu = 0,
\]
Finally, remember we are dealing with $f+ \delta$ instead of $f$, then
\begin{equation}\label{Chen3}
\int_X F'_i(0)d\mu = \int_X (f_i -f -\delta)\log (f+\delta) + o (\delta),
\end{equation} 
and the limit will converge to zero when $\delta\rightarrow 0$, since
\begin{eqnarray}
\lim_{i}\int_X F'(0)d\mu &= & \int_{X}(f-f-\delta)\log(f+\delta) + o(\delta)
\nonumber\\
&=& o(\delta),
\end{eqnarray}
which completes the proof.
\end{proof}
\medskip

\noindent We treat next a version of the previous lemma, which  
will be very useful later on.

Let $\chi$ be a continuous function; we define $h_A = \exp(\chi - A)$, and we consider the following truncated version of the entropy functional 
\[
H_A(\varphi): = \int_X f_i \log\max (f_i, h_A) d\mu.
\]
\medskip

\noindent Our claim is as follows. 

\begin{lemma}\label{Chen4}
The truncated entropy functional has the following semi-continuity type property 
\[
\lim_{i} H_A(\varphi_i) - H_A(\varphi) \geq -\delta(A) 
\]
for some uniform constant $\delta(A)>0$ such that $\delta(A)\rightarrow 0$ as $A$ tends to $\infty$. 
Here the sequence $(f_i)$ is assumed to verify the hypothesis of the preceding lemma.
\end{lemma}
\begin{proof}
First define 
$$\tilde{H}_A(\varphi) = \int_X \max(f, h_A)\log\max(f, h_A) d\mu. $$
then set $\Omega_A = \{ f < h_A \}$, and $\Omega_{i,A} = \{f_i < h_A  \}$
\[
H_A(\varphi) - \tilde{H}_A(\varphi) = \int_X \big(f - \max(f, h_A) \big) \log\max(f, h_A) d\mu
\]
\[
=\int_{\Omega_A} (f - h_A) \log h_A d\mu,
\]
then notice that $0> f - h_A \geq -h_A$ on $\Omega_A$ and $\log h_A < 0$ for $A$ large enough,
\[
| H_A(\varphi) - \tilde{H}_A(\varphi) | \leq -\int_X (\chi - A) e^{\chi -A} d\mu = C_1(A).
\]
Now run the same argument as in lemma (\ref{Chen}) with $f$ replaced by $\max(f, h_A)$(here we have positive lower bound automatically from this truncation), then obtain
\[
H(\varphi_i) - \tilde{H}_A(\varphi) \geq \int_X F'_i(0) d\mu,
\]  
from equation (\ref{Chen1}), then by equation (\ref{Chen2}) it's enough to estimate 
\begin{equation}\label{Chen3}
\int_X \big( f_i - \max(f, h_A)\big) \log \big( \max(f, h_A) + 1 \big) d\mu
\end{equation}
\[
= \int_{X - \Omega_A} (f_i - f) \log(f+1) d\mu + \int_{\Omega_A} (f_i - h_A) \log(h_A +1) d\mu,
\]
the first term will converges to zero as $i\rightarrow +\infty$ as before, and the second term is bounded from below by
\[
\int_{\Omega_A \cap \Omega_{i, A}} (f_i - h_A) \log(h_A +1) d\mu \geq -\int_{\Omega_A \cap \Omega_{i, A}} h_A \log(1 + h_A) d\mu
\]
\[
\geq -\int_X h_A\log(1+ h_A) d\mu
\]
\[
\geq -2 \int_X e^{\chi - A} d\mu = -C_2(A).
\]
Finally observe that  
\[
H_A(\varphi_i) - H(\varphi_i) = \int_X f_i \big( \log\max(f_i, h_A) - \log f_i \big) \geq 0, 
\]
hence we can decompose 
\[
H_A(\varphi_i) - H_A(\varphi) = \big(H_A(\varphi_i) - H(\varphi_i)  \big) +\big( H(\varphi_i) -  \tilde{H}_{A}(\varphi) \big) 
\]
\[
+ \big(  \tilde{H}_A(\varphi) - H_A(\varphi)  \big),
\]
then the limit of $H_A(\varphi_i) - H_A(\varphi)$ is bounded below by 
$$-\delta(A):= -C_1(A)-C_2(A),$$ which converges to zero when $A$ is large. 
\end{proof}

\medskip

\noindent The Mabuchi functional can be written as  
\begin{equation}\label{2008}
\cM(\varphi) = E(\varphi) + H(\varphi)
\end{equation}
where $E$ is the energy part of the Mabuchi functional, and $H$ is the entropy part. As a consequence of our previous results, we 
establish here the following statement.

\begin{theorem}\label{MT}
The Mabuchi functional $\cM(\varphi(t))$ is a continuous convex function on $[0,1]$. 
\end{theorem}
\begin{proof}
The proof results immediately from our previous considerations. Indeed, since $\cM$ is a convex function, we automatically have
\begin{equation}\label{2108}
\cM(0)\geq \lim\sup_{t\to 0} \cM(t)
\end{equation}
On the other hand, the semi-continuity properties of the entropy functional in Lemma \ref{Chen}
show that in fact we have 
\begin{equation}\label{2208}
\cM(0)\leq \lim\inf_{t\to 0} \cM(t)
\end{equation}
and the proof of Theorem \ref{MT} is completed. Indeed, the \emph{energy} part of the Mabuchi functional is continuous, given the regularity
of the potential of $\cG$.
\end{proof}




\section{Almost convexity along $\ep$-geodesics}

In \cite{XX}, the geodesic $\cG$ is obtained by the continuity method, and as a by-product 
of the proof,
for each $\ep> 0$ one has a smooth, positive (1,1)-form
\begin{equation}\label{equa601}
\omega_\ep:= \omega+ dd^c \rho_\ep
\end{equation}
on $X\times \Sigma$ such that the next identity holds
\begin{equation}\label{equa602}
\omega_{\ep}^{n+1}= \ep \sqrt{-1}dt\wedge d\ol t\wedge \omega
\end{equation}

\noindent Let $\cM_{\ep, A}:\Sigma\to \bR$ be the regularization of the Mabuchi functional 
evaluated on $\omega_\ep$. By definition, this equals

\begin{equation}\label{equa2006}
\cM_{\ep, A}(t) = \frac{S}{n+1} \cE(\varphi_t) - \cE^{Ric_{\omega}} (\rho_{\ep, t}) + 
\int_X \max \Big(\log\frac{\omega_\ep^n}{\omega^n}, \log\frac{h_A}{\omega^n}\Big)\omega_\ep^n
\end{equation}
where $h_A$ is a volume element on $X$ whose associated curvature is greater than $-C\cG$ for some positive constant $C$. Then we can compute the Hessian of $\cM_{\ep, A}$ on the approximate geodesic as 
\begin{eqnarray}
\label{mabuchi6}
\int_{X\times\Sigma}\cM_{\ep,A}  d_{t}d^c_{t}\tau &=& \int_{X\times \Sigma} \tau (\omega^{n+1}_{\ep} - Ric(\omega)\wedge \omega_{\ep}^n)
\nonumber\\
&+& \int_{X\times\Sigma}  \log\max\big(\frac{\omega^n_{\ep}}{\omega^n}, \frac{h_A}{\omega^n}\big) \omega^n_{\ep}\wedge d_{t}d^c_{t}\tau
\nonumber\\
&=& \ep - \int_{X\times\Sigma} Ric(\omega)\wedge \omega_{\ep}^n + \int_{X\times \Sigma} \tau dd^c \log\max\big(\frac{\omega^n_{\ep}}{\omega^n}, \frac{h_A}{\omega^n}\big) \omega^n_{\ep}.
\end{eqnarray}
where $\tau(t)$ is any test function on $\Sigma$.
\medskip

\noindent Now we have the following result.
\begin{theorem}\label{ac}
For each positive constant $A$, there is a uniform constant $C_A>0$ such that the function $\cM_{\ep, A} - \ep C_A t(1-t )$ is convex. 
\end{theorem}

\begin{proof} It would be enough to prove that for any $A\gg 0$ there exists a constant $C_A$ such that we have
\begin{eqnarray}\label{equa510}
\int_{X\times \Sigma}\tau dd^c\max 
\Big(\log\frac{\omega_\varepsilon^n}{\omega_0^n}, \log\frac{h_A}{\omega_0^n}\Big)\wedge \omega_\varepsilon^n & \geq & \int_{X\times \Sigma}\tau\Ricci_{\omega_0}\wedge  \omega_{\ep}^n \nonumber \\
&- & \ep C_A\int_{\Sigma}\tau \sqrt{-1} dt\wedge d\ol t
\end{eqnarray}
for any positive test function $\tau$ on $\Sigma$. Indeed, once this is done 
we infer that the second  convexity of $\cM_\ep$ follows. The inequality (\ref{equa510}) is established 
in the next paragraph by a direct computation.

\medskip

\noindent We recall the following statement, which will be useful during the arguments 
in the following section.

\begin{lemma}
Let $u$ and $v$ be two smooth functions on a complex manifold $Z$, and let $\omega$ be a K\"ahler metric on 
$Z$. We assume that $v$ is subharmonic with respect to $\omega$, and that we equally have
$\displaystyle \Delta_\omega(u)\geq 0$ on the set $u> v-1$. Then the $\omega$-Laplacian of the function $\max(u, v)$ is positive.
\end{lemma}

\noindent Indeed this follows from the fact that a smooth function is subharmonic if and only if it satisfies
the mean value inequality  (is this context, the usual Lebesgue measure is replaced with the harmonic measure on balls); we refer to \cite{GW} and the references therein for a complete account of these facts. 
In the next section we will have to deal with functions whose
Laplacian is greater than $-C$. Then we have a similar statement, since locally we can construct functions with strictly positive Laplacian (e.g. the potential of the metric $\omega$).
\medskip

\subsection{The computation}

\noindent In order to simplify the notations, we define the function $f_\varepsilon: X\times \Sigma\to \bR$ by the equality
\begin{equation}\label{equa511}
\frac{\omega_\varepsilon^n}{\omega_0^n}\big|_{X\times \{t\}}= e^{f_\varepsilon(t, \cdot)}.
\end{equation}
We introduce the set 

\begin{equation}\label{equa512}
\Omega_{\varepsilon, A}:= \Big\{(z, t)\in X\times \Sigma \hbox{ such that } f_\varepsilon(t, z)> \log\frac{h_A}{\omega_0^n}(t, z)\Big\}
\end{equation}
so that we have
\begin{equation}\label{equa513}
\max \Big(\log\frac{\omega_\varepsilon^n}{\omega_0^n}, \log\frac{h_A}{\omega_0^n}\Big)= f_\varepsilon(t, z)
\end{equation}
on $\Omega_{\varepsilon, A}$. This reveals the importance of considering the functional $\cM_{\ep, A}$: on the set 
$\Omega_{\varepsilon, A}$ the distortion function $f_\varepsilon$ is bounded from below \emph{by a quantity which is independent of $\varepsilon$}. This simple remark will play a crucial role in the next considerations.

We will proceed next to the evaluation of the integral
\begin{equation}\label{equa514}
\int_{\Omega_{\varepsilon, A}}\tau \big(dd^cf_\varepsilon- \Ricci_\omega\big)\wedge \omega_\varepsilon^n.
\end{equation}
Locally near a point $(z, t)\in X\times \Sigma$ we write the metric $\omega_\varepsilon$ as follows
\begin{eqnarray}\label{equa515} 
\omega_\varepsilon &= & \sqrt{-1}g_{t\ol t}dt\wedge d\ol t+ \sqrt{-1}g_{t\ol \alpha}dt\wedge dz^{\ol \alpha}+ 
\sqrt{-1}g_{\alpha\ol t} dz^{\alpha}\wedge d\ol t \nonumber \\
&+ & \sqrt{-1}g_{\gamma\ol \alpha}dz^{\gamma}\wedge dz^{\ol \alpha} \\
\nonumber
\end{eqnarray}
where the coefficients $g$ in the expression above depend on $\varepsilon$ as well. 

The equation (\ref{equa511}) satisfied by $\omega_\varepsilon$ can be written in local coordinates as
\begin{eqnarray}\label{equa516}
c(\varphi_\varepsilon)\!\!\!&:= & g_{t\ol t}- g^{\gamma \ol \alpha}g_{\gamma \ol t}g_{t\ol \alpha}\nonumber \\
&=& \varepsilon e^{-f_{\varepsilon}}.\\
\nonumber
\end{eqnarray}
Given this, we rewrite locally the metric $\omega_\varepsilon$ as follows
\begin{equation}\label{equa517}
\omega_\varepsilon= c(\varphi_\ep)\sqrt{-1}dt\wedge d\ol t+ \rho_\ep
\end{equation}
where $\displaystyle\rho_\ep$ has the same expression as $\omega_\ep$, except that we replace
$\displaystyle g_{t\ol t}$ with $\displaystyle g^{\gamma \ol \alpha}g_{\gamma \ol t}g_{t\ol \alpha}$. We note that although $\rho_\ep$ may not be closed, it is positive definite on each slice $X\times \{t\}$ and it satisfies
\begin{equation}\label{equa518}
\rho_\ep^{n+1}= 0.
\end{equation} 

We have the equality
\begin{equation}\label{equa519}
\omega_\ep^n= \rho_\ep^n+ nc(\varphi_\ep)\sqrt{-1}dt\wedge d\ol t\wedge \rho_\ep^{n-1}
\end{equation}
which is the same as 
\begin{equation}\label{equa520}
\omega_\ep^n= \rho_\ep^n+ nc(\varphi_\ep)\sqrt{-1}dt\wedge d\ol t\wedge \omega_\ep^{n-1}
\end{equation}
by the definition of the form $\rho_\varepsilon$.

The factor 
\begin{equation}\label{equa521}
nc(\varphi_\ep)\big(dd^cf_\varepsilon- \Ricci_\omega\big)\wedge \sqrt{-1}dt\wedge d\ol t\wedge \omega_\ep^{n-1}
\end{equation}
is analyzed as follows.

We observe that we have
\begin{equation}\label{equa522}
nc(\varphi_\ep)dd^cf_\varepsilon\wedge \sqrt{-1}dt\wedge d\ol t\wedge \omega_\ep^{n-1}= 
c(\varphi_\ep)\Delta_{\omega_\varepsilon}(f_\varepsilon)\sqrt{-1}dt\wedge d\ol t\wedge \omega_\ep^{n}
\end{equation}
and by the equality (\ref{equa516}) we have 
\begin{equation}\label{equa523}
nc(\varphi_\ep)dd^cf_\varepsilon\wedge \sqrt{-1}dt\wedge d\ol t\wedge \omega_\ep^{n-1}= \ep 
\Delta_{\omega_\varepsilon}(f_\varepsilon)\sqrt{-1}dt\wedge d\ol t\wedge \omega_0^{n}
\end{equation}
The other term in the equality (\ref{equa521}) is bounded in $L^1$ norm by $\varepsilon C_A$, given that 
on the set $\Omega_{\ep, A}\cap X\times \{t\}$ the eigenvalues of $\omega_\ep$ are bounded from below (and above) by a constant independent of
$\varepsilon$, so that the trace of $\Ricci_\omega$ with respect to $\omega_\ep$ is bounded by some constant $C_A$. 

Therefore, we have
\begin{equation}\label{equa524}
nc(\varphi_\ep)\frac{\big(dd^cf_\varepsilon- \Ricci_\omega\big)\wedge \sqrt{-1}dt\wedge d\ol t\wedge \omega_\ep^{n-1}}{\sqrt{-1}dt\wedge d\ol t\wedge \omega_0^{n}}
\geq \ep 
\Delta_{\omega_\varepsilon}(f_\varepsilon)- \ep C_A 
\end{equation}
This can be re-written in the following way
\begin{eqnarray}\label{equa554}
nc(\varphi_\ep)\big(dd^cf_\varepsilon- \Ricci_\omega\big)\wedge \sqrt{-1}dt\wedge d\ol t\wedge \omega_\ep^{n-1}
& \geq & \!\!\!
\Delta_{\omega_\varepsilon}(f_\varepsilon)\omega_\ep^{n+1} \nonumber \\
&-& \!\!\! C_A\omega_\ep^{n+1} 
\end{eqnarray}
\medskip

\noindent The evaluation of the main term 
\begin{equation}\label{equa525}
\big(dd^cf_\varepsilon- \Ricci_\omega\big)\wedge \rho_\varepsilon^n
\end{equation}
goes as follows. Let 
\begin{equation}\label{equa526}
v:= \frac{\partial}{\partial t}- g^{\gamma\ol \alpha}g_{t\ol \alpha}\frac{\partial}{\partial z^{\gamma}}
\end{equation}
be the gradient of the $t$ derivative of $\varphi_\ep$. Then one can check by a direct computation that 
the vector field $v$ generates the kernel of $\rho_\ep$, and then we have  

\begin{equation}\label{equa527}
\big(dd^cf_\varepsilon- \Ricci_\omega\big)\wedge \rho_\varepsilon^n= 
\big(dd^cf_\varepsilon- \Ricci_\omega\big)(v, \ol v)\sqrt{-1}dt\wedge d\ol t\wedge \rho_\varepsilon^n.
\end{equation}
Indeed, this is a matter of liner algebra: we have 
$$\beta_1\wedge \beta_2^n= \frac{\beta_1(v, \ol v)}{\lambda(v, \ol v)}\lambda\wedge \beta_2^n$$
on a vector space of dimension $n+1$, where $\lambda, \beta_j$ are (1,1)-forms, such that $v$ is in the 
kernel of $\beta_2$, and such that $\lambda(v, \ol v)\neq 0$.   
\smallskip

A straightforward calculation which we will detail in a moment shows that we have 
\begin{equation}\label{equa528}
\big(dd^cf_\varepsilon- \Ricci_\omega\big)(v, \ol v)\geq  
\ep \Delta_{\omega_\varepsilon} \big(e^{-f_\ep}\big)
\end{equation}
Since we have $\displaystyle \Delta_{\omega_\varepsilon} \big(e^{-f_\ep}\big)\geq -e^{-f_\ep}\Delta_{\omega_\varepsilon}(f_\varepsilon)$, the inequality (\ref{equa528}) combined with (\ref{equa554}) finishes the proof.
Indeed, we first remark that we have 
$$\sqrt{-1}dt\wedge d\ol t\wedge \rho_\varepsilon^n= \sqrt{-1}dt\wedge d\ol t\wedge \omega_\varepsilon^n;$$
by (\ref{equa528}) we obtain
\begin{equation}\label{equa558}
\big(dd^cf_\varepsilon- \Ricci_\omega\big)(v, \ol v)\geq -
\ep \Delta_{\omega_\varepsilon}(f_\varepsilon)\sqrt{-1}dt\wedge d\ol t\wedge \omega_0^n
\end{equation}
and we observe that the right hand side of (\ref{equa558}) is nothing but
$$-\Delta_{\omega_\varepsilon}(f_\varepsilon)\omega_\varepsilon^{n+1}. $$ 
Thus, we infer the inequality
\begin{equation}\label{equa550}
dd^c\max \Big(\log\frac{\omega_\varepsilon^n}{\omega_0^n}, \log\frac{h_A}{\omega_0^n}\Big)\wedge \omega_\varepsilon^n\geq - C_A \omega_\ep^{n+1}
\end{equation}
globally on $X\times \Sigma$.

\smallskip

\noindent We prove next the inequality (\ref{equa528}); before that, we remark that the following approach is
quite standard in the theory of the  homogeneous Monge-Amp\`ere equations, cf. \cite{XX},
\cite{Chen2}, \cite{CT}...
Also, the inequality (\ref{equa528}) is very similar to the positivity of the curvature along the leaves of the foliation (which does not exists in our case...), cf. \cite{CT}.

The next computations are done with respect to a geodesic coordinate system at $(X, z)$; we have
\begin{equation}\label{equa529}
\ol\partial \log\det(g_{\alpha\ol\beta})= g^{\alpha\ol \beta}g_{\alpha \ol \beta, \ol t}d\ol t
+ g^{\alpha\ol \beta}g_{\alpha \ol \beta, \ol \gamma}dz^{\ol \gamma}
\end{equation}
and thus
\begin{eqnarray}\label{equa530}
\ddbar\log\det(g_{\alpha\ol\beta})&=& \big(g^{\alpha\ol \beta}_{, t}g_{\alpha \ol \beta, \ol t}+ 
g^{\alpha\ol \beta}g_{\alpha \ol \beta, t\ol t}\big)
dt\wedge d\ol t\nonumber\\
&+ & g^{\alpha\ol \beta}g_{\alpha \ol \beta, \gamma \ol t}dz^\gamma \wedge d\ol t+ 
g^{\alpha\ol \beta}g_{\alpha \ol \beta, t \ol \gamma}  dt\wedge dz^{\ol \gamma} \\
& + &
g^{\alpha\ol \beta}g_{\alpha \ol \beta, \gamma\ol \tau}
dz^{\gamma}\wedge dz^{\ol \tau} \nonumber\\
\nonumber
\end{eqnarray} 
Since the metric $\omega_\ep$ is locally given by the Hessian of a function, the following commutation relations
\begin{equation}\label{equa531}
g_{\alpha \ol \beta, t\ol t}= g_{t\ol t, \alpha \ol \beta}
\end{equation}
hold true on $X$. Given that 
\begin{equation}\label{equa532}
g^{\alpha\ol \beta}_{, t}= -g^{\alpha\ol \gamma}g^{\delta\ol \beta}g_{\delta \ol \gamma, t}
\end{equation}
the equality (\ref{equa530}) become
\begin{eqnarray}\label{equa533}
\ddbar\log\det(g_{\alpha\ol\beta})&=& \big(g^{\alpha\ol \beta}g_{ t\ol t, \alpha \ol \beta}-  
g^{\alpha\ol \gamma}g^{\delta\ol \beta}g_{\delta\ol \gamma, \ol t}g_{\alpha\ol \beta, t}
\big)
dt\wedge d\ol t\nonumber\\
&+ & g^{\alpha\ol \beta}g_{\alpha \ol \beta, \gamma \ol t}dz^\gamma \wedge d\ol t+ 
g^{\alpha\ol \beta}g_{\alpha \ol \beta, t \ol \gamma}  dt\wedge dz^{\ol \gamma} \\
& + &
g^{\alpha\ol \beta}g_{\alpha \ol \beta, \gamma\ol \tau}
dz^{\gamma}\wedge dz^{\ol \tau}. \nonumber\\
\nonumber
\end{eqnarray}

We evaluate this in the $v$-direction, and we get
\begin{eqnarray}\label{equa534}
\ddbar\log\det(g_{\alpha\ol\beta})(v, \ol v)&=& g^{\alpha\ol \beta}g_{ t\ol t, \alpha \ol \beta}-  
g^{\alpha\ol \gamma}g^{\delta\ol \beta}g_{\gamma \ol \delta, \ol t}g_{\alpha\ol \beta, t}
\nonumber\\
&- & g^{\alpha\ol \beta}g_{\alpha \ol \beta, \gamma \ol t}g^{\gamma\ol\mu}g_{t\ol \mu}-
g^{\alpha\ol \beta}g_{\alpha \ol \beta, t \ol \gamma} g^{\mu\ol\gamma}g_{\mu \ol t} \\
& + &
g^{\alpha\ol \beta}g_{\alpha \ol \beta, \gamma\ol \tau}g^{\gamma\ol \mu}g^{\rho\ol \tau}g_{t\ol \mu}g_{\rho \ol t}. \nonumber\\
\nonumber
\end{eqnarray}

\smallskip

The equation satisfied by the metric $\omega_\varepsilon$ reads as 
\begin{equation}\label{equa535}
g_{t\ol t}- g^{p\ol q}g_{p\ol t}g_{t\ol q}= \ep e^{-f_\ep}
\end{equation}
so that we have
\begin{eqnarray}\label{equa536}
g^{\alpha\ol \beta}g_{ t\ol t, \alpha \ol \beta}- \ep\Delta_{\omega_\ep}(e^{-f_\ep})& = & 
g^{\alpha\ol \beta}g^{p\ol q}_{, \alpha\ol \beta}g_{p\ol t}g_{t\ol q}\nonumber \\
&+& g^{\alpha\ol \beta}g^{p\ol q}g_{p\ol t, \alpha\ol \beta}g_{t\ol q}+ 
g^{\alpha\ol \beta}g^{p\ol q}g_{p\ol t}g_{t\ol q, \alpha\ol \beta} \\
&+ & g^{\alpha\ol \beta}g^{p\ol q}g_{p\ol t, \alpha}g_{t\ol q, \ol \beta}+ 
g^{\alpha\ol \beta}g^{p\ol q}g_{p\ol t, \ol \beta}g_{t\ol q, \alpha}\nonumber \\
\nonumber
\end{eqnarray}
By combining (\ref{equa536}) with (\ref{equa534}) we obtain
\begin{equation}\label{equa37}
\ddbar\log\det(g_{\alpha\ol\beta})(v, \ol v)= |\dbar v|^2+ \ep\Delta_{\omega_\ep}(e^{-f_\ep})
\end{equation}
and the inequality (\ref{equa528}) follows.
\end{proof}

\subsection{Further results and comments}
It is very likely that the convexity of $\cM$ in the sense of distributions can be derived 
by the techniques we have developed in the previous section, i.e. using $\ep$--geodesics. One of the motivations to do so 
is that the resulting proof would be more 
``self-contained''. 

However, we encounter a rather severe difficulty: we ignore whether the fiber-wise sequence of volume elements 
$$\omega^n_{\ep}$$
corresponding to the $\ep$-geodesics is converging almost everywhere to the volume element of the geodesic $\cG$.

Nevertheless, we strongly believe that this holds true, based on the following considerations. On the set 
$\displaystyle \Omega_{\ep, A}$ we have 
\begin{equation}\label{equa601}
C(A)\omega< \omega_\ep< C\omega
\end{equation}
where $C(A)$ is a constant depending on $A$, but uniform with respect to $\ep$, and $C$ is a fixed constant, independent of $\ep, A$.
Indeed this is a consequence of the results in \cite{XX}. The relation (\ref{equa601}) is a uniform 
\emph{Laplacian estimate} for the metrics $\displaystyle \omega_\ep|_{\Omega_\ep}$. Hence, via Evans-Krilov theory one might hope that 
it is possible to obtain a higher regularity estimate for the family $\displaystyle (\varphi_\ep|_{\Omega_{\ep, A}})_{\ep> 0}$. 
The problem is that
as $\ep\to 0$, the set $\Omega_{\ep, A}$ converges eventually towards a set which is only measurable, 
and it is a-priori unclear how to implement Evans-Krilov theory in this setting.
 \\

\noindent However, we show here that the continuity of $\cM$ at the endpoints 0 and 1 can be also obtained as a consequence of the 
results we have established in the previous section (i.e. without knowing a-priori the convexity of $\cM$).

\begin{theorem}\label{continuity}
The Mabuchi functional $\cM(t)$ is continuous at the boundary points 0 and 1. 
\end{theorem}
\begin{proof}
We identify in what follows $\tau$ and its real part $t=Re(\tau)$, since all the functionals involved 
in the proof only depends on the real part of $\tau$. Also, we will only prove the continuity at 0.

A first observation is that we have
\begin{equation}\label{equa602}
\lim_{t\to 0}\cM(t)\geq \cM(0)
\end{equation}
thanks to the entropy property recalled in \ref{Chen}.

Next, we observe that the $\limsup$ of a sequence of convex functions which are locally bounded above is still convex 
(unlike subharmonic functions). Hence if we define 
\[
\limsup_{\ep\rightarrow 0} \cM_{\ep, A}:= \cM_{A},
\]
then $\cM_A$ is a convex function on $[0,1]$ by theorem (\ref{ac}). 
And by construction we have $\cM_A(0) = \cM(0)$ for any value of the regularization parameter $A$.

Now for every point $\tau\in (0,1)$, we have $\cM_{A}(\tau)\geq \cM(\tau) - \delta(A)$, since 
\[
\limsup_{\ep\rightarrow 0} H_A(\varphi_{\ep}) \geqslant H_A(\varphi) - \delta(A),
\] 
by lemma (\ref{Chen4}). 

We define a new functional 
\[
\limsup_{A\rightarrow +\infty}\cM_A := \widetilde{\cM}, 
\]
and then $t\to \widetilde{\cM}(t)$ is a convex function on $[0,1]$ which still verifies the equality 
$\widetilde{\cM}(0) = \cM(0)$. 

Then we have $\widetilde{\cM}(0)\geq \lim_{t\to 0}\widetilde{\cM}(t)$ by convexity, as well as the inequality
$\widetilde{\cM}(\tau)\geq \cM(\tau)$ for each $\tau\in (0,1)$, thanks to the considerations above.
We therefore infer that 
\begin{equation}\label{equa603}
\lim_{t\to 0}\cM(t)\leq \cM(0)
\end{equation}
and Theorem \ref{continuity} is proved.

\end{proof}

\end{document}